\documentclass{article}
\usepackage[centertags]{amsmath}
\usepackage{amsfonts}
\usepackage{amssymb}
\usepackage{amsthm}
\usepackage{newlfont}
\usepackage[top=1in, bottom=1in, left=1in, right=1in]{geometry}
\usepackage{bbold}

\theoremstyle{plain}
\newtheorem{theorem}{Theorem}
\newtheorem*{ihara}{Ihara's Theorem}
\newtheorem{proposition}{Proposition}
\newtheorem{corollary}{Corollary}
\newtheorem{lemma}{Lemma}

\theoremstyle{definition}
\newtheorem{definition}{Definition}
\newtheorem{example}{Example}

\newcommand{\R}{\mathbb{R}}

\newcommand{\C}{\mathbb{C}}

\newcommand{\Hidden}[1]{}

\DeclareMathOperator{\vol}{vol}
\DeclareMathOperator{\Ker}{Ker}
\DeclareMathOperator{\Image}{Image}

\begin{document}

\title{Non-backtracking random walks and a weighted Ihara's theorem}
\author{Mark Kempton\footnote{Center of Mathematical Sciences and Applications, Harvard University, Cambridge, MA.  Email: mkempton@cmsa.fas.harvard.edu}}

\date{}

\maketitle

\begin{abstract}
We study the mixing rate of non-backtracking random walks on graphs by looking at non-backtracking walks as walks on the directed edges of a graph.  A result known as Ihara's Theorem relates the adjacency matrix of a graph to a matrix related to non-backtracking walks on the directed edges.  We prove a weighted version of Ihara's Theorem which relates the transition probability matrix of a non-backtracking walk to the transition matrix for the usual random walk.  This allows us to determine the spectrum of the transition probability matrix of a non-backtracking random walk in the case of regular graphs and biregular graphs.  As a corollary, we obtain a result of Alon et. al. in \cite{alon} that in most cases, a non-backtracking random walk on a regular graph has a faster mixing rate than the usual random walk.  In addition, we obtain an analogous result for biregular graphs.  
\end{abstract}

%-----------------------------------------------------------------------------------------------------------------------------------------------------
\section{Introduction}

A random walk on a graph $G$ is a random process on the vertices of $G$ in which, at each step in the walk, we choose uniformly at random among the neighbors of the current vertex.  Random walks have been studied extensively, and are used in a variety of algorithms involving graphs.  For a comprehensive survey on random walks on graphs, see \cite{lovasz}, and for applications of spectral techniques to random walk theory, see \cite{fan}.  Random walks on graphs have the useful property that given any initial distribution on the vertex set, the random walk converges to a unique stationary distribution as long as the graph is connected and not bipartite.  The speed at which this convergence takes place is referred to as the \emph{mixing rate} of the random walk.  In a graph where a random walk has a fast mixing rate, vertices can be sampled quickly using this random process, making this a useful tool in theoretical computer science.  

A non-backtracking random walk on a graph is a random walk with the added condition that, on a given step, we are not allowed to return to the vertex visited on the previous step.  Viewed as a walk on vertices, a non-backtracking random walk loses the property of being a Markov chain, making its analysis somewhat more difficult.  However, their study has received increased interest in recent years.  Recently, Angel, Friedman, and Hoory \cite{angel} studied non-backtracking walks on the universal cover of a graph.  Fitzner and Hofstad \cite{Fitzner} studied the convergence of non-backtracking random walks on lattices and tori.  Krzakala et. al. \cite{redemption} use a matrix related to non-backtracking walks to study spectral clustering algorithms.  Most pertinent to the current paper, Alon, Benjamini, Lubetzky, and Sodin \cite{alon} studied the mixing rate of a non-backtracking walk for regular graphs.  In particular, they prove that in most cases, a non-backtracking random walk on a regular graph has a faster mixing rate than a random walk allowing backtracking.  

In this paper, we study the mixing rate for a non-backtracking random walk, with the goal of removing the condition of regularity needed in the results of Alon et. al. in \cite{alon}.  We take a different approach than Alon et. al. by looking at the non-backtracking walk as a walk along directed edges of a graph, as is done in \cite{angel}.  This allows us to turn the non-backtracking random walk into a Markov chain, but on a larger state space, which in turn allows us to determine the stationary distribution to which a non-backtracking walk converges for a general graph, whether or not it is regular.  In the case of regular graphs, our approach allows us to compute the spectrum of the transition probability matrix for a non-backtracking random walk, expressed in terms of the eigenvalues of the adjacency matrix.  This allows for easy comparison of the mixing rates of a non-backtracking random walk, and an ordinary random walk.  As a corollary, this gives us an alternate proof of the result in \cite{alon} for regular graphs.  Our approach gives more information than the approach in \cite{alon}, in that the full spectrum of the transition probability matrix is given.  In addition, we are able to compute the spectrum of the non-backtracking transition probability matrix for biregular graphs.  As a corollary, we generalize the result in \cite{alon} for regular graphs to an analogous result for biregular graphs. 

A key component in our proof is a weighted version of a result known as Ihara's Theorem, also called the Ihara zeta identity, which relates an operator indexed by the directed edge set of a graph to an operator indexed by the vertex set of the graph.  Ihara's Theorem was first considered in the study of number theoretic zeta functions on graphs, and was first proved for regular graphs by Ihara in 1966 (see \cite{ihara}).  Numerous other proofs have been given since, along with generalizations to irregular graphs, by Hashimoto (\cite{hashimoto}, 1989), Bass (\cite{bass}, 1992), Stark and Terras (\cite{stark}, 1996), Kotani and Sunada (\cite{zeta}, 2000), and others.  We will give an elementary proof of Ihara's Theorem that, to our knowledge, is original.  In addition, we follow ideas similar to those in \cite{zeta} to obtain a version of Ihara's Theorem with weights that allows us to study the relevant transition probability matrices for random walks.  

The remainder of this paper is organized as follows.  In section \ref{sec:prelim}, we give the necessary background and preliminary information on random walks, and develop the corresponding theory for non-backtracking walks, including the convergence of a non-backtracking walk to a stationary distribution for a general graph.  We accomplish this via walks on the directed edges of a graph.  We also investigate bounds obtained from the normalized Laplacian for a directed graph.  We also give the relevant background on Ihara's Theorem, and a new elementary proof. In section \ref{sec:ihara}, we prove our weighted version of Ihara's formula.  Finally, in section \ref{sec:rate}, we use this formula to obtain the spectrum of the transition probability matrix for a non-backtracking random walk for regular and biregular graphs.  This gives a new proof of the result of Alon et. al. concerning the mixing rate of a non-backtracking random walk on a regular graph, and generalizes this result to the class of biregular graphs.  
%-----------------------------------------------------------------------------------------------------------------------------------------------------

\section{Preliminaries}\label{sec:prelim}

\subsection{Random walks}

Throughout this paper, we will let $G = (V,E)$ denote a graph with vertex set $V$ and (undirected) edge set $E$, and we will let $n = |V|$ and $m = |E|$. A \emph{random walk} on a graph is a sequence $(v_0,v_1,...,v_k)$ of vertices $v_i\in V$ where $v_i$ is chosen uniformly at random among the neighbors of $v_{i-1}$.  Random walks on graphs are well-studied, and considerable literature exists about them.  See in particular \cite{fan} and \cite{lovasz} for good surveys, especially in the use of spectral techniques in studying random walks on graphs.

The \emph{adjacency matrix} $A$ of $G$ is the $n\times n$ matrix with rows and columns indexed by $V$ given by
\[
A(u,v) = \begin{cases}1&\text{ if } u\sim v\\ 0 & \text{ otherwise.} \end{cases}
\] 
It is a well-known fact that the $(u,v)$ entry of $A^k$ is the number of walks of length $k$ starting at vertex $u$ and ending at vertex $v$.  Define $D$ to be the $n\times n$ diagonal matrix with rows and columns indexed by $V$ with $D(v,v) = d_v$, where $d_v$ denotes the degree of vertex $v$.  A random walk on a graph $G$ is a Markov process with transition probability matrix $P = D^{-1}A$, so \[P(u,v) = \begin{cases}\frac{1}{d_u}&\text{ if }u\sim v\\0&\text{ otherwise. }\end{cases}\]  Given any starting probability distribution $f_0$ on the vertex set $V$, the resulting expected distribution $f_k$ after applying $k$ random walk steps is given by $f_k = f_0P^k$.  Here we are considering $f_0$ and $f_k$ as row vectors in $\R^n$.

 Note that, in general, $P$ is not symmetric for an irregular graph, but is similar to the symmetric matrix $D^{-1/2}AD^{-1/2}$.  Thus, the eigenvalues of $P$ are real, and if we order them as $\mu_1\geq \mu_2\geq\cdots\geq\mu_n$, then it is easy to see that $\mu_1 = 1$ with eigenvector $\mathbf1$, and $\mu_n \geq -1$.  By Perron-Frobenius theory, if the matrix $P$ is irreducible, then we have that $\mu_2 <1$, and if $P$ is aperiodic, then $\mu_n > -1$.  The matrix $P$ being irreducible and aperiodic corresponds to the graph $G$ being connected and non-bipartite.  

The \emph{stationary distribution} for a random walk on $G$ is given by \[\pi(v) = \frac{d_v}{\vol(G)}.\]  The stationary distribution has the important property the $\pi P = \pi$, so that a random walk with initial distribution $\pi$ will stay at $\pi$ at each step.  An important fact about the stationary distribution is that if $G$ is a connected graph that is not bipartite, then for any initial distribution $f_0$ on $V(G)$, we have
\[
\lim_{t\rightarrow\infty} (f_0P^t)(v) = \pi(v)
\]
for all $v$ (see \cite{lovasz}).  

Knowing that a random walk will converge to some stationary distribution, a fundamental question to consider is to determine how quickly the random walk approaches the stationary distribution, or in other words, to determine the \emph{mixing rate}.  In order to make this question precise, we need to consider how to measure the distance between two distribution vectors.  

Several measures for defining the mixing rate of a random walk have been given (see \cite{fan}).  Classically, the mixing rate is defined in terms of the pointwise distance (see \cite{lovasz}).  That is, the mixing rate is
\[
\rho = \limsup_{t\rightarrow\infty}\max_{u,v}\left| P^t(u,v) - \pi(v)\right|^{1/t}.
\]
Note that a small mixing rate corresponds to fast mixing.  Alternatively, the mixing rate can be considered in terms of the standard $L_2$ (Euclidean) norm, %\[\|fP^t - \pi\|,\]
 the \emph{relative pointwise distance},
%\[
%\Delta(t) = \max_{u,v}\frac{|P^t(u,v) - \pi(v)|}{\pi(v)},
%\]
the \emph{total variation distance},
%\[
%\Delta_{TV}(t) = \max_{A\subset V(G)}\max_{u\in V(G)} \left|\sum_{v\in A}(P^t(u,v) - \pi(v)\right|,
%\]
or the \emph{$\chi$-squared distance}.
%\[
%\Delta'(t) = \max_{u\in V(G)}\left(\sum_{v\in V(G)} \frac{(P^t(u,v) - \pi(v))^2}{\pi(v)}\right)^{1/2}.
%\]
In general, these measures can yield different distances, but spectral bounds on the mixing rate are essentially the same for each.  See \cite{fan} for a detailed comparison of each.  For our purposes, we will primarily be concerned with the $\chi$-squared distance, which will be defined below.

The mixing rate of a random walk is directly related to the eigenvalues of $P$.
\begin{theorem}[Corollary 5.2 of \cite{lovasz}]
Let $G$ be a connected non-bipartite graph with transition probability matrix $P$, and let the eigenvalues of $P$ be $1=\mu_1 > \mu_2 \geq \cdots \geq \mu_n > -1$.  Then the mixing rate is $\max\{\mu_2,|\mu_n|\}$.
\end{theorem}
Thus, the smaller the eigenvalues of $P$, the faster the random walk converges to its stationary distribution.

\subsection{Non-backtracking random walks}\label{subsec:nb_walk}

  A \emph{non-backtracking random walk} on $G$ is a sequence $(v_0, v_1,...,v_k)$ of vertices $v_i\in V$ where $v_{i+1}$ is chosen randomly among the neighbors of $v_i$ such that $v_{i+1} \neq v_{i-1}$ for $i=1,...,k-1$.  In other words, a non-backtracking random walk is a random walk in which a step is not allowed to go back to the immediately previous state.   A non-backtracking random walk on a graph is not a Markov chain since, in any given state, we need to remember the previous step in order to take the next step.  In order for this to be well-defined, we assume throughout the remainder of the paper that the minimim degree of $G$ is at least 2.  

Define $P^{(k)}$ to be the $n\times n$ transition probability matrix for a $k$-step non-backtracking random walk on the vertices.  That is $P^{(k)}(u,v)$ is the probability that a non-backtracking random walk starting at vertex $u$ ends up at vertex $v$ after $k$ steps.  Note that $P^{(1)} = P$, where $P=D^{-1}A$ is the transition matrix for an ordinary random walk on $G$.  However, $ P^{(k)}$ is not simply $P^k$ since a non-backtracking random walk is not a Markov chain.

This process can be turned into a Markov chain, however, by changing the state space from the vertices of the graph to the directed edges of the graph.  That is, replace each edge in $E$ with two directed edges (one in each direction). Then the non-backtracking random walk is a sequence of directed edges $(e_1,e_2,\cdots,e_k)$ where if $e_i=(v_j,v_k)$, and $e_{i+1}=(v_r,v_s)$ then $v_k=v_r$ and $v_s\neq v_j$.  That is, the non-backtracking condition restricts the walk from moving from an edge to the edge going in the opposite direction. Denote the set of directed edges by $\overrightarrow E$. The transition probability matrix for this process we will call $\tilde P$.  Observe that
\[
\tilde P( (u,v),(x,y) ) = \begin{cases} \frac{1}{d_v-1} & \text{ if } v=x \text{ and } y\neq u\\ 0 & \text{ otherwise}.\end{cases}
\]
Note that $\tilde P$ is a $2m\times 2m$ matrix.  Note also that $\tilde P^k$ is the transition matrix for a walk with $k$ steps on the directed edges.

\begin{lemma}
Given any graph $G$, the matrix $\tilde P$ as defined above is doubly stochastic.
\end{lemma}
\begin{proof}
Observe first that the rows of the matrix $\tilde P$ sum to 1, as it is a transition probability matrix.  In addition, the columns of $\tilde P$ sum to 1.  To see this, consider the column indexed by the directed edge $(u,v)$.  The entry of this column corresponding to the row indexed by $(x,y)$ is $\frac{1}{d_y-1}$ if $y=u$ and if $v\neq x$.  Since $y=u$ this is equal to $\frac{1}{d_u-1}$. Otherwise, the entry is 0.  Thus the column sum is \[ \sum_{\substack{x\sim u\\x\neq v}} \frac{1}{d_u-1} = \frac{d_u-1}{d_u-1} =1\] as claimed.
\end{proof}

Define the distribution $\tilde\pi :\overrightarrow E \rightarrow \R$ by
\[
\tilde\pi = \frac{\mathbb 1}{\vol(G)}
\]
where $\mathbb 1$ is the vector of length $2m$ with each entry equal to 1.

\begin{lemma}\label{lem:edge_stationary}
Let $\tilde f_0:\overrightarrow E \rightarrow \R$ be any distribution on the directed edges of $G$.  If the matrix $\tilde P$ is irreducible and aperiodic, then
\[
\tilde f_0\tilde P^k \longrightarrow \tilde\pi
\]
as $k\rightarrow \infty$.
\end{lemma}
\begin{proof}
It follows from Lemma 1 that $\tilde\pi$ is a stationary distribution for $\tilde P$.  This follows because, since the columns of $\tilde P$ sum to 1, we have
\[
\tilde\pi\tilde P = \tilde\pi.
\]
Therefore, if the sequence $\tilde f_0\tilde P^k$ converges, it must converge to $\tilde\pi$.  Now, $\tilde P$ being irreducible and aperiodic are precisely the conditions for this to converge.
\end{proof}

Let $f$ be a probability distribution on the vertices of $G$.  Then $f$ can be turned into a distribution $\tilde f$ on $\overrightarrow E$ as follows.  Define \[\tilde f((u,v)) = \frac{1}{d_u} f(u).\]  Conversely, given a distribution $\tilde g$ on $\overrightarrow E$, define a distribution $g$ on the vertices by \[g(u) = \sum_{(u,v)\in\overrightarrow E} \tilde g(u,v).\]

Thus, given any starting distribution $f_0:V\rightarrow \R$ on the vertex set of $G$, we can compute the distribution after $k$ non-backtracking random walk steps $f_k:V\rightarrow\R$ as follows.  First compute the distribution $\tilde f_0$ on the directed edges as above, then compute $\tilde f_k = \tilde f_0\tilde P^k$, then $f_k$ is given by $f_k(u) = \sum_{v\sim u} \tilde f_k(u,v)$.  The following proposition tells us that this converges to the same stationary distribution as an ordinary random walk on a graph.

\begin{theorem}
Given a graph $G$ and a starting distribution $f_0:V\rightarrow \R$ on the vertices of $G$, define $f_k = f_0P^{(k)}$ to be the distribution on the vertices after $k$ non-backtracking random walk steps.  Define the distribution $\pi:V\rightarrow \R$ by $\pi(v) = \frac{d_v}{\vol(G)}$ (note that this is the stationary distribution for an ordinary random walk on $G$). Then if the matrix $\tilde P$ is irreducible and aperiodic, then for any starting distribution $f_0$ on $V$, we have
\[
f_k \longrightarrow \pi \text{ as } k\rightarrow\infty.
\]
\end{theorem}
\begin{proof}
As described above, take the distribution $f_0$ on vertices to the corresponding distribution $\tilde f_0$ on directed edges.  Then define $\tilde f_k = \tilde f_0\tilde P^k$.  Then by Lemma \ref{lem:edge_stationary}, $\tilde f_k$ converges to $\tilde\pi$.  Now $\tilde\pi = \frac{\mathbb 1}{\vol G}$, and observe that \[\pi(u) = \frac{d_u}{\vol(G)} = \sum_{v\sim u} \frac{1}{\vol(G)} = \sum_{v\sim u} \tilde\pi((u,v)).\]  So pulling the distribution $\tilde\pi$ on directed edges back to a distribution on the vertices yields $\pi$.  Thus the result follows.
\end{proof}

%%%%%%%%%%%%%%%%%%%%%%%%%%%%%%%%%%%%%%%%%%%%%%%%%%%%%%%%

\begin{definition}
The \emph{$\chi$-squared distance} for measuring convergence of a random walk is defined by
\[
\Delta'(t) = \max_{y\in V(G)}\left( \sum_{x\in V(G)} \frac{(\tilde P^t(y,x) - \tilde \pi(x))^2}{\tilde\pi(x)}\right)^{1/2}.
\]
\end{definition}

Notice that since $\tilde \pi = \mathbb1/\vol(G)$,
\[\begin{split}
\Delta'(t)^2 &= \max_y \frac{1}{2m}\| (\chi_y\tilde P^t - \tilde \pi)\|^2 \\  &=\max_y \frac{1}{2m}\| (\chi_y - \tilde \pi)\tilde P^t\|^2
\end{split}\]

\begin{theorem}
Let $\mu_1=1, \mu_2,\cdots,\mu_{2m}$ be the eigenvalues of $\tilde P$.  Then the convergence rate for the non-backtracking random walk with respect to the $\chi$-squared distance is bounded above by $\max_{i\neq1} |\mu_i|.$
\end{theorem}
\begin{proof}
We have
\[
\Delta'(t)^2 = \max_y \frac{1}{2m}\| (\chi_y - \tilde \pi)\tilde P^t\|^2.
\]
Observe that $\chi_u -\tilde\pi$ is orthogonal $\tilde \pi$, which is the eigenvector for $\mu_1$, so we see that
\[
\Delta'(t) \leq \frac{1}{2m}\max_{i\neq1} |\mu_i|^t.
\]
Therefore,
\[
\lim_{t\rightarrow\infty}(\Delta'(t))^{1/t}\leq\max_{i\neq1}|\mu_i|.
\]

\end{proof}

%~~~~~~~~~~~~~~~~~~~~~~~~~~~~~~~~~~~~~~~~~~~~~~~~~~~~~~~~~
\subsection{Non-backtracking Walks as Walks on a Directed Graph}\label{sub:dir_laplacian}

The transition probability matrix $\tilde P$ for the walk on directed edges can be thought of as a transition matrix for a random walk on a directed line graph of the graph $G$.  In this way, theory for random walks on directed graphs can be applied to analyze non-backtracking random walks.  Random walks on directed graphs have been studied by Chung in \cite{directed} by way of a directed version of the normalized graph Laplacian matrix.  In \cite{directed}, the Laplacian for a directed graph is defined as follows.  Let $P$ be the transition probability matrix for a random walk on the directed graph, and let $\phi$ be its Perron vector, that is, $\phi P = \phi$.  Then let $\Phi$ be the diagonal matrix with the entries of $\phi$ along the diagonal.  Then the \emph{Laplacian} for the directed graph is defined as 
\[
\mathcal L = I - \frac{\Phi^{1/2}P\Phi^{-1/2} + \Phi^{-1/2}P^*\Phi^{1/2}}{2}.
\]
This produces a symmetric matrix that thus has real eigenvalues.  Those eigenvalues are then related to the convergence rate of a random walk on the directed graph.  In particular, the convergence rate is bounded above by 
$2\lambda_1^{-1}(-\log\min_x\phi (x))$, where $\lambda_1$ is the second smallest eigenvalue of $\mathcal L$ (see Theorem 7 of \cite{directed}).

Applying this now to non-backtracking random walks, define $\tilde P$ as before.  Then as seen above, $\phi$ is the constant vector with $\phi(v) = 1/\vol(G)$ for all $v$.  Then the directed Laplacian for a non-backtracking walk becomes 
\[
\tilde{ \mathcal L} = I_{2m} = \frac{\tilde P + \tilde P^*}{2}.
\]
Then Theorem 1 of \cite{directed}, applied to the matrix $\tilde{\mathcal L}$ as defined, gives the Rayleigh quotient for a function $f:\overrightarrow E\rightarrow\C$ by
\[
\tilde R(f) = \frac{f^*\tilde{\mathcal L}f}{f^*f} =\frac12\frac{\sum\limits_{(u,v)\in\overrightarrow E(G)}\, \sum\limits_{\substack{(v,w) \\ w\neq u}}\left(f(u,v) -  f(v,w)\right)^2\tilde P((u,v),(v,w))}{\sum\limits_{(u,v) \in \overrightarrow E(G)}f(u,v)^2}.
\]
From this it is clear that $\tilde{\mathcal L}$ is positive semidefinite with smallest eigenvalue $\lambda_0 = 0$.  If $0=\lambda_0\leq\lambda_1\leq\cdots\leq\lambda_{2m-1}$ are the eigenvalues of $\tilde{\mathcal L}$, then Theorem 7 from \cite{directed} implies that the convergence rate for the corresponding random walk is bounded above by
\[
\frac{2\log\vol(G)}{\lambda_1}.
\]

We remark that for an ordinary random walk on an undirected graph $G$, the convergence rate is also on the order of $1/\lambda_1(\mathcal L)$, where $\mathcal L$ now denotes the normalized Laplacian of the undirected graph $G$.  Note that
\[
\lambda_1(\mathcal L) = \inf_{\substack{f:V(G)\rightarrow \R\\ f \perp D\mathbb1}}R(f)
\]
where $\displaystyle R(f) = \frac{\sum_{uv\in E(G)}(f(u) - f(v))^2}{\sum_{v\in V(G)}f(v)^2d_v}$ denotes the Rayleigh quotient with respect to $\mathcal L$, and
\[
\lambda_1(\tilde{\mathcal L}) = \inf_{\substack{ f:\overrightarrow E(G)\rightarrow \R \\ f\perp\mathbb1 }}\tilde R(f)
\]
with $\tilde R$ given above.

The following result shows that the Laplacian bound does not give an improvement for non-backtracking random walks over ordinary random walks.  

\begin{proposition}
Let $G$ be any graph, and let $\mathcal L$ be the normalized graph Laplacian and $\tilde{\mathcal L}$ the non-backtracking Laplacian defined above.  Then we have
\[
\lambda_1(\tilde{\mathcal L}) \leq \lambda_1(\mathcal L).
\]
\end{proposition}
\begin{proof}
Let $f:V(G)\rightarrow \R$ be the function orthogonal to $D\mathbb1$ that achieves the minimum in the Rayleigh quotient for $\mathcal L$.  So
\[ \sum_{v\in V(G)} f(v)d_v = 0 \, \text{  and  } \, \lambda_1(\mathcal L) = \frac{\sum_{uv\in E(G)}(f(u) - f(v))^2}{\sum_{v\in V(G)}f(v)^2d_v}.\]
Define $ f':\overrightarrow E \rightarrow \R$ by $f'(u,v) = f(u)$.  Observe that
\[
\sum_{(u,v)\in\overrightarrow E(G)}  f'(u,v) = \sum_{(u,v)\in\overrightarrow E(G)} f(u) = \sum_{u\in V(G)}f(u)d_u = 0.
\]
So $f'$ is orthogonal to $\mathbb1$.  Therefore
\[\begin{split}
\lambda_1(\tilde{\mathcal L}) \leq \tilde R( f') &= \frac12\frac{\sum\limits_{(u,v)\in\overrightarrow E(G)}\, \sum\limits_{\substack{(v,w) \\ w\neq u}}\left(f'(u,v) -  f'(v,w)\right)^2\tilde P((u,v),(v,w))}{\sum\limits_{(u,v) \in \overrightarrow E(G)}f'(u,v)^2}\\
&=\frac12\frac{ \sum\limits_{(u,v)}\sum\limits_{\substack{(v,w)\\w\neq u }} \left(f(u) - f(v)\right)^2\frac{1}{d_v-1}}{\sum\limits_{(u,v)}f(u)^2}\\
&=\frac12\frac{\sum\limits_{(u,v)}\left(f(u) - f(v)\right)^2}{\sum\limits_{u\in V(G)} f(u)^2d_u}\\
&= \frac{\sum\limits_{\{u,v\}\in E(G)}\left(f(u) - f(v)\right)^2}{\sum\limits_{u\in V(G)}f(u)^2d_u} = R(f) = \lambda_1(\mathcal L).
\end{split}\]
\end{proof}

%~~~~~~~~~~~~~~~~~~~~~~~~~~~~~~~~~~~~~~~~
\subsection{Ihara's Theorem}

The transition probability matrix $\tilde P$ defined above is a weighted version of an important matrix that comes up in the study of zeta functions on finite graphs.  We define $B$ to be the $2m\times 2m$ matrix with rows and columns indexed by the set of directed edges of $G$ as follows.
\[
B( (u,v),(x,y) ) = \begin{cases} 1 & \text{ if } v=x \text{ and } y\neq u\\ 0 & \text{ otherwise}.\end{cases}
\]

The matrix $B$ can be thought of as a non-backtracking edge adjacency matrix, and the entries of $B^k$ describe the number of non-backtracking walks of length $k$ from one directed edge to another, in the same way that the entries of powers of the adjacency matrix, $A^k$, count the number of walks of length $k$ from one vertex to another. The expression $\det(I-uB)$ is closely related to zeta functions on finite graphs which.  A result known as Ihara's Theorem further relates such zeta functions to a determinant expression involving the adjacency matrix.  While we will not go into zeta functions on finite graphs in this paper, the following result equivalent to Ihara's theorem will be of interest to us.  \\

\begin{ihara}  For a graph $G$ on $n$ vertices and $m$ edges, let $B$ be the matrix defined above, let $A$ denote the adjacency matrix, $D$ the diagonal degree matrix, and $I$ the identity.  Then
\[
\det(I-uB) = (1-u^2)^{m-n}\det(I-uA+u^2(D-I)).
\]
\end{ihara}

  We remark that the expression $\det(I-uB)$ is the characteristic polynomial of $B$ evaluated at $1/u$.  In this way the complete spectrum of the matrix $B$ is given by the reciprocals of the roots of the polynomial $(1-u^2)^{m-n}\det(I-uA+u^2(D-I))$.  Numerous proofs of this result exist in the literature \cite{ihara,hashimoto,bass,stark,zeta}.  For completeness, we will include here an elementary proof that uses only basic linear algebra.  To the knowledge of the author, this proof is original.  To begin, we will need a lemma giving a well-known property of determinants.

\begin{lemma}\label{lem:order}
Let $M$ be a $k\times l$ matrix, $N$ a $l\times k$ matrix, and $A$ an invertible $k\times k$ matrix. Then 
\[
\det(A+MN) = \det(A)\det(I+NA^{-1}M).
\]
\end{lemma}
\begin{proof}
Note that
\[
\begin{bmatrix}I&N\\0&I \end{bmatrix}\begin{bmatrix}I&0\\A^{-1}M&A^{-1}(A+MN)\end{bmatrix}\begin{bmatrix}I &-N \\ 0&I\end{bmatrix} = \begin{bmatrix}I+NA^{-1}M&0\\A^{-1}M&I\end{bmatrix}.
\]
Taking determinants of both sides gives the result.
\end{proof}
\noindent{\it Proof of Ihara's Theorem.}  Define $S$ to be the $2m\times n$ matrix
\[
S((u,v),x) = \begin{cases} 1 & \text{ if } v=x \\ 0 & \text{ otherwise}\end{cases}
\]
so $S$ is the endpoint incidence operator.  Define $T$ to be the $n\times 2m$ matrix given by
\[
T(x,(u,v)) = \begin{cases} 1& \text{ if } u=x \\ 0 & \text{ otherwise}\end{cases}
\]
so $T$ is the starting point incidence operator.  We will also define $\tau$ to be the $2m\times 2m$ matrix giving the reversal operator that switches a directed edge with its opposite.  That is,
\[
\tau((a,b),(c,d)) = \begin{cases} 1 & \text{ if } b=c, a=d \\ 0 & \text{ otherwise}\end{cases}
\]

Now, a straightforward computation verifies that 

\begin{equation}\label{eq:B}
B = ST - \tau,
\end{equation}
\begin{equation}\label{eq:A}
A = TS,
\end{equation}
and
\begin{equation}\label{eq:D}
D = T\tau S.
\end{equation}
 
Then from Lemma \ref{lem:order} and (\ref{eq:B}) we obtain

\[\begin{split}
\det(I-uB) &= \det(I-u(ST-\tau))\\
&= \det(I+u\tau - uST)\\
&= \det(I+u\tau)\det(I-uT(I+u\tau)^{-1}S)
\end{split}\]
where $u$ is chosen so that the matrix $I+u\tau$ is inverivle.  

Observe that $\tau^2 = I$, so that $(I-u\tau)(I+u\tau) = (1-u^2)I$, so $(I+u\tau)^{-1} = \frac{1}{1-u^2}(I-u\tau)$.  Thus, applying (\ref{eq:A}) and (\ref{eq:D}), the above becomes 

\[\begin{split}
\det(I-uB) &= \det(I+u\tau)\det(I-\frac{u}{1-u^2}T(I-u\tau)S)\\
&=\det(I+u\tau)\det(I-\frac{u}{1-u^2}(TS - uT\tau S))\\
&= \det(I+u\tau)\frac{1}{(1-u^2)^n}\det((1-u^2)I - uA + u^2D)\\
&=(1-u^2)^{m-n}\det(I - uA + u^2(D-I))
\end{split}\]
where the last step is obtained by observing that $\det(I+u\tau) = (1-u^2)^m$.  This is the desired equality for our choice of $u$.  This is a polynomial of finite degree in $u$, and there are infinitely many $u$ that make $I+u\tau$ invertible, so the equality holds for all $u$. 

\qed

%-----------------------------------------------------------------------------------------------------------------------------------------------------

%-----------------------------------------------------------------------------------------------------------------------------------------------------
\section{A weighted Ihara's theorem}\label{sec:ihara}
  In this section, we will give a weighted version of Iharra's Theorem.  The proof presented in the previous section does not lend itself well to generalization to the weighted setting, so we will not follow that strategy.  Rather, we will follow the main ideas of the proof of Ihara's theorem found in \cite{zeta} to obtain our weighted version of this result.

To each vertex $x\in V(G)$ we assign a weight $w(x)\neq0$, and let $W$ be the $n\times n$ diagonal matrix given by $W(x,x) = w(x)$. Define $S$ and $T$ to be the matrices from the proof of Ihara's Theorem in the previous section, and define $\tilde S = SW$ and $\tilde T = WT$.  So $\tilde S$ is the weighted version of the endpoint vertex-edge incidence operator, and $\tilde T$ is the weighted version of the starting point vertex-edge incidence operator.  Define $\tau$ from the proof of Ihara's Theorem, 
% to be the $2m \times n$ matrix whose rows are indexed by the directed edges of $G$ and whose columns are indexed by the %vertices of $G$, given by
%\[
%S((u,v),x) = \begin{cases} 1 & \text{ if } v=x \\ 0 & \text{ otherwise}\end{cases}
%\]
%and define $\tilde S = SW$.  The matrix $S$ is the vertex-edge incidence matrix for endpoints of directed edges, and $\tilde S$ is %the weighted version of $S$.  Define $T$ to be the $n\times 2m$ matrix given by
%\[
%T(x,(u,v)) = \begin{cases} 1& \text{ if } u=x \\ 0 & \text{ otherwise}\end{cases}
%\]
%and define $\tilde T= WT$.  The matrix $T$ is the vertex-edge incidence operator for starting points of directed edges.  We will %also define $\tau$ to be the $2m\times 2m$ matrix giving the reversal operator that switches a directed edge with its opposite.  %That is,
%\[
%\tau((a,b),(c,d)) = \begin{cases} 1 & \text{ if } b=c, a=d \\ 0 & \text{ otherwise}\end{cases}
%\]
and define $\tilde\tau$ to be the weighted version of $\tau$, that is
\[
\tilde\tau((a,b),(c,d)) = \begin{cases} w(b)^2 & \text{ if } b=c, a=d \\ 0 & \text{ otherwise}\end{cases}
\]
Finally, define the $2m\times 2m$ matrix $\tilde P$ by
\begin{equation}\label{eq:Ptilde}
\tilde P((a,b),(c,d)) =  \begin{cases} w(b)^2 & \text{ if } b=c, a\neq d \\ 0 & \text{ otherwise.}\end{cases}
\end{equation}
Then $\tilde P$ is the weighted version of the non-backtracking edge adjacency matrix $B$ seen above in Ihara's theorem, with $w(b)^2$ the weight on edge $(a,b)$.  We remark that if we take $w(x) = 1/\sqrt{d_x-1}$ for each $x\in V(G)$, then $\tilde P$ is exactly the transition probability matrix for a non-backtracking random walk on the directed edges of $G$ defined in Section \ref{subsec:nb_walk}.  This case is our primary focus, but we note that our computations apply for any arbitrary positive weights assigned to the vertices.

Now, a straightforward computation verifies that

\begin{equation}\label{eq:ST}
\tilde P = \tilde S\tilde T - \tilde\tau
\end{equation}
and
\begin{equation}\label{eq:TS}
\tilde T\tilde S = WAW.
\end{equation}
We will define $\tilde A = WAW$.  Note that $\tilde A(u,v) = w(u)w(v)$, so this is the adjacency matrix for the weighted graph with edge weights $w(u)w(v)$.  The matrix $\tilde A$ is similar to $W^2 A$, so when $w(x) = 1/\sqrt{d_x-1}$, this is the matrix whose entries are the transition probabilities for a single step of a non-backtracking random walk $G$.  

From (\ref{eq:ST}) and (\ref{eq:TS}) we obtain the following equations.
\begin{eqnarray}
(I-u\tilde P)(I-u\tilde\tau)& = I-u\tilde S\tilde T+u^2\tilde S\tilde T\tilde\tau-u^2\tilde\tau^2\label{eq:I-P}\\
(I-u\tilde\tau)(I-u\tilde P) &= I-u\tilde S\tilde T+u^2\tilde\tau\tilde S\tilde T-u^2\tilde\tau^2
\end{eqnarray}

We define $\tilde D$ to be the diagonal $n\times n$ matrix $\tilde D(x,x) = \sum_{v\sim x}w(x)^2w(v)^2$ and observe that $\tilde T\tilde\tau \tilde S = \tilde D$.  It then follows that
\begin{eqnarray}
\left((I-u\tilde P)(I-u\tilde \tau)+u^2\tilde\tau^2\right)\tilde S &=\tilde S\left(I - u\tilde A+u^2\tilde D\right) \label{eq:S}\\
\tilde T\left((I-u\tilde\tau)(I-u\tilde P)+u^2\tilde\tau^2\right) &= \left(I - u\tilde A+u^2\tilde D\right)\tilde T \label{eq:T}
\end{eqnarray}

We remark that in the proof in \cite{zeta}, they use the unweighted versions of each of these matrices, so $\tau$ rather than $\tilde\tau$ yields $\tau^2= I$. Hence $S$ and $T$ will factor through $\tau^2$, so that the $u^2\tau^2$ term stays on the right hand side of the above equations.  Here we have $\tilde\tau^2$ is a $2m\times 2m$ diagonal matrix with $\tilde\tau^2((u,v),(u,v)) = w(u)^2w(v)^2$.  Depending on the $w(u)$'s this matrix might not behave nicely with respect to the action of $S$ and $T$, hence the extra terms that need to stay on the left-hand side above.  This difference from \cite{zeta} is one of the primary difficulties in generalizing this result.

\Hidden{%%%%%%%%%%%%%%%%%%%%%%%%%%%%%%%%%%%%%%%%%%%%%%%%%%%%%%%%%%%%%%%%%%%%%%%%%%%%%
We will now need the following Lemma.

\begin{lemma}\label{lemma:direct_sum}
 The $2m$-dimensional space of functions $f:\overrightarrow E \rightarrow \C$ from the set of directed edges of $G$ to the complex numbers can be expressed as the direct sum of the subspaces
\[\Image \tilde S \bigoplus \Ker \tilde T\tilde\tau.\]
  %In particular, these are complementary invariant subspaces of the operator $(I-u\tilde P)(I-u\tilde\tau)+u^2\tilde\tau^2$.
\end{lemma}
\begin{proof}
Observe that, by construction, the operators $\tilde S$ and $\tilde T$ both have rank $n$, and $\tau$ is bijective (recall that $w(x) = 0$ for all $x\in V(G)$.  Therefore

\[
\dim \Image \tilde S + \dim \Ker\tilde T\tilde\tau = \Image \tilde S+\dim\Ker \tilde T = n + (2m-n) = 2m.
\]

It remains to check that $\Image S\cap\Ker T\tilde\tau = \{0\}$.  Consider $g = \tilde Sf \in \Image\tilde S$, and note $\tilde T\tilde\tau g = \tilde T\tilde\tau\tilde S f $.  Recall that $\tilde T\tilde\tau\tilde S = \tilde D$ is a diagonal matrix with non-zero entries along the diagonal, so we see that $g \in \Ker\tilde T\tilde \tau$ if and only  $g = 0$.   So we have the direct sum decomposition as claimed in the lemma.

\end{proof}
}%%%%%%%%%%%%%%%%%%%%%%%%%%%%%%%%%%%%%%%%%%%%%%%%%%%%%%%%%%%%%%%%%%%%%%%%%%%%%%%%%%%

We will now perform a change of basis to see how the operator $(I-u\tilde P)(I-u\tilde\tau)+u^2\tilde\tau^2$ behaves with respect to the decomposition of the space of functions  $f:\overrightarrow E \rightarrow \C$ as the direct sum of $\Image \tilde S$ and $\Ker \tilde S^T$.  To this end, fix any basis of the subspace $\Ker \tilde S^T$, and let $R$ be the $2m\times(2m-n)$ matrix whose columns are the vectors of that basis (note that $\tilde S$ has rank $n$).  Define $\displaystyle M = \begin{bmatrix}\tilde S & R\end{bmatrix}$.  This will be our change of basis matrix.  To obtain the inverse of $M$, form the matrix $\displaystyle \begin{bmatrix} (\tilde S^T\tilde S)^{-1}\tilde S^T\\(R^TR)^{-1}R^T\end{bmatrix}$ and observe that
\[
\begin{bmatrix}(\tilde S^T\tilde S)^{-1}\tilde S^T\\(R^TR)^{-1}R^T\end{bmatrix}\begin{bmatrix}\tilde S & R\end{bmatrix} = \begin{bmatrix} (\tilde S^T\tilde S)^{-1}\tilde S^T\tilde S & (\tilde S^T\tilde S)^{-1}\tilde S^T R\\(R^TR)^{-1}R^T\tilde S& (R^TR)^{-1}R^TR\end{bmatrix} = \begin{bmatrix}I_n&0\\0&I_{2m-n}\end{bmatrix}.
\]
Therefore we have that $M^{-1} = \displaystyle \begin{bmatrix}(\tilde S^T\tilde S)^{-1}\tilde S^T\\(R^TR)^{-1}R^T\end{bmatrix}$.

Applying this change of basis, direct computation, applying (\ref{eq:I-P}) and (\ref{eq:S}), yields

\begin{equation}\label{eq:decomp}
\begin{bmatrix}(\tilde S^T\tilde S)^{-1}\tilde S^T\\(R^TR)^{-1}R^T\end{bmatrix}\left((I-u\tilde P)(I-u\tilde\tau)+u^2\tilde\tau^2\right)\begin{bmatrix}\tilde S& R\end{bmatrix} = \begin{bmatrix}I-u\tilde A +u^2\tilde D&-u\tilde T R + u^2\tilde T\tilde\tau R\\0&I  \end{bmatrix}.
\end{equation}

Therefore, the matrix $(I-u\tilde P)(I-u\tilde\tau)+u^2\tilde\tau^2$ is similar to the matrix $\begin{bmatrix}I-u\tilde A +u^2\tilde D&-u\tilde T R + u^2\tilde T\tilde\tau R\\0&I  \end{bmatrix}$, so they have the same determinant.  Thus, we have proven a weighted version of Ihara's Theorem, which we state as the following.

\begin{theorem}\label{thm:weighted_ihara}
Let $G$ be a graph on $n$ vertices and $m$ edges, and assign an arbitrary positive weight $w(x)>0$ assigned to each vertex $x$. Let $\tilde P$ be the $2m\times 2m$ weighted non-backtracking edge adjacency matrix with edge weight $w(v)^2$ assigned to edge $(u,v)$ as defined in (\ref{eq:Ptilde}).  Let $\tilde A$ be the weighted $n\times n$ adjacency matrix with edge weight $w(u)w(v)$ assigned to each edge. Let $\tilde \tau$ be the weighted reversal operator defined above, and $\tilde D$ the $n\times n$ diagonal matrix with $\tilde D(x,x) = \sum_{v\sim x}w(x)^2w(v)^2$ as defined above. Then we have
\[
\det\left((I-u\tilde P)(I - u\tilde \tau) + u^2\tilde\tau^2\right) = \det( I - u \tilde A + u^2\tilde D).
\]
\end{theorem}

As a corollary to the decomposition in equation (\ref{eq:decomp}), if we take $w(x) = 1$ for all $x$, then $\tilde\tau^2 = I$, and the usual unweighted Ihara's Theorem falls out immediately.  

If we take $w(x) = \frac{1}{\sqrt{d_x-1}}$, then $\tilde P$ becomes the transition probability matrix for the non-backtracking walk on directed edges, and $\tilde A(u,v) = \frac{1}{\sqrt{(d_u-1)(d_v-1)}}$.  This is clearly similar to the matrix $(D-I)^{-1}A$.  So in this case $\tilde A$ is similar to the matrix whose entries are the transition probabilities for a single step in a non-backtracking random walk.  (Note, however, that $(D-I)^{-1}A$ is not the transition probability matrix for a non-backtracking random walk.)

%-----------------------------------------------------------------------------------------------------------------------------------------------------
\section{The mixing rate of non-backtracking random walks}\label{sec:rate}

\subsection{An alternate proof for regular graphs}

Applying the results of the previous section to regular graphs yields a different proof of the results from \cite{alon} on the mixing rate of non-backtracking random walks on regular graphs.

Let $G$ be a regular graph where each vertex has degree $d$.  Then choosing $w(x)  = 1/\sqrt{d-1}$ for all $x$ yields gives us that $\tilde P$ is the transition probability matrix for the non-backtracking random walk on $G$.  We remark that, from the previous section, we have $\tilde \tau = \frac{1}{d-1}\tau$, $\tilde \tau^2 = \frac{1}{(d-1)^2}I$, $\tilde A = \frac{1}{d-1}A$, and $\tilde D = \frac{d}{(d-1)^2}I$.  Therefore, the decomposition in (\ref{eq:decomp}) becomes
\begin{equation*}
(I-u\tilde P)(I-u\tilde\tau) \sim \begin{bmatrix}I-\frac{u}{d-1} A +\frac{u^2}{d-1}I&*\\0&\left(1-\frac{u^2}{(d-1)^2}\right)I  \end{bmatrix}.
\end{equation*}

Noting that $\tilde \tau$ can be thought of as block diagonal with $m$ blocks of the form $\begin{bmatrix} 0 &1/(d-1)\\ 1/(d-1) &0 \end{bmatrix}$, then taking determinants, we find that
\[
\det(I-u\tilde P) \left(1-\frac{u^2}{(d-1)^2}\right)^m = \left(1-\frac{u^2}{(d-1)^2}\right)^{2m-n}\det\left(I-\frac{u}{d-1}A + \frac{u^2}{d-1}I\right)
\]
and hence
\[
\det(I-u\tilde P) = \left(1-\left(\frac{u}{d-1}\right)^2\right)^{m-n}\prod_{i=1}^n\left(1-\frac{\lambda_i}{d-1} u + \frac{1}{d-1}u^2\right)
\]
where the product ranges over all the eigenvalues $\lambda_i$ of the adjacency matrix $A$ for $i=1,\cdots,n$.  As remarked previously, the left hand side $\det(I-u\tilde P)$ is the characteristic polynomial of $\tilde P$ evaluated at $1/u$, so from this we obtain the spectrum of $\tilde P$.
\begin{theorem}
Let $G$ be a $d$-regular graph with $m$ edges and $n$ vertices, and let $\tilde P$ be the $2m\times 2m$ transition probability matrix for a non-backtracking random walk as defined above.  Then the eigenvalues of $\tilde P$ are
\[
\pm\frac{1}{d-1}, \, \frac{\lambda_i\pm\sqrt{\lambda_i^2 - 4(d-1)}}{2(d-1)}, \, (i = 1,\cdots,n)
\]
where $\lambda_i$ ranges over the eigenvalues of the adjacency matrix $A$, and $\pm1/(d-1)$ each have multiplicity $m-n$.  
\end{theorem}

From this we obtain the result from \cite{alon}.

\begin{corollary}\label{thm:alon}
Let $G$ be a non-bipartite, connected $d$-regular graph on $n$ vertices for $d\geq3$, and let $\rho$ and $\tilde \rho$ denote the mixing rates of simple and non-backtracking random walk on $G$, respectively.  Let $\lambda$ be the second largest eigenvalue of the adjacency matrix of $G$ in absolute value.  

If $\lambda \geq 2\sqrt{d-1}$, then 
\[
\frac{d}{2(d-1)} \leq \frac{\tilde \rho}{\rho}\leq 1.
\]

If $\lambda < 2\sqrt{d-1}$ and $d=n^{o(1)}$, then
\[
\frac{\tilde \rho}{\rho} = \frac{d}{2(d-1)} + o(1).
\]
\end{corollary}

\begin{proof}
We remark that the expression  $ \frac{\lambda+\sqrt{\lambda^2 - 4(d-1)}}{2(d-1)}$ is precisely the expression derived by Alon et al. in \cite{alon} for the mixing rate of a non-backtracking random walk on a regular graph, and we may proceed with the analysis of the convergence rate in the same way they do.  The convergence rate is given by the second largest eigenvalue of $\tilde P$, which will be obtained setting $\lambda$ to be the second largest eigenvalue of $A$.  Let $\mu$ be this eigenvalue. 

For $2\sqrt{d-1}\leq \lambda \leq d$ we have
\[
\frac{\lambda}{2(d-1)} < \frac{\lambda + \sqrt{\lambda^2 - 4(d-1)}}{2(d-1)} \leq \frac{\lambda}{d}.
\]
So $\frac{\lambda}{2(d-1)} \leq \mu\leq \frac{\lambda}{d}$.  Since $\frac{\lambda}{d}$ is the second largest eigenvalue of the transition probability matrix $P$ for the usual walk, the first case follows.  

For $\lambda < 2\sqrt{d-1}$, $\mu$ is complex, and we obtain
\[
|\mu|^2 = \left|\frac{\lambda + \sqrt{\lambda^2 - 4(d-1)}}{2(d-1)}\right|^2 = \left(\frac{\lambda}{2(d-1)}\right)^2 + \left(\frac{\sqrt{4(d-1) - \lambda^2 }}{2(d-1)}\right)^2  = \frac{1}{d-1}
\]
so $|\mu| = \frac{1}{\sqrt{d-1}}$.

We remark that in this case that $\lambda < 2\sqrt{d-1}$, a classic result of Nilli (\cite{nilli}) related to the Alon-Boppana Theorem implies that we are never too far below this bound.  Indeed, the result states that if $G$ is $d$-regular with diameter at least $2(k+1)$, then $\lambda \geq 2\sqrt{d-1} - \frac{2\sqrt{d-1}-1}{k+1}$.  With the restriction that $d = n^{o(1)}$, then the diameter is at least $(1-o(1))\log_{d-1}n$, and so $\lambda > (1-o(1))2\sqrt{d-1}$, and the second case follows.   
\end{proof}

%~~~~~~~~~~~~~~~~~~~~~~~~~~~~~~~~~~~~~~
\subsection{Biregular graphs}
A graph $G$ is called \emph{$(c,d)$-biregular} if it is bipartite and each vertex in one part of the bipartition has degree $c$, and each vertex of the other part has degree $d$.  In the weighted Ihara's Theorem, we have $\tilde \tau^2 ((u,v),(u,v) = \frac{1}{(d_u-1)(d_v-1)}$, so in the case where $G$ is $(c,d)$-biregular, then we have $\tilde\tau^2 = \frac{1}{(c-1)(d-1)}I$.  So since $\tilde\tau^2$ is a multiple of the identity, as with regular graphs, in the decomposition (\ref{eq:decomp}), the $u^2\tilde\tau^2$ term can be taken to the other side of the equation.  Note that $\tilde D$ is diagonal with $\tilde D(u,u) = \sum_{v\sim u}\frac{1}{(d_u-1)(d_v-1)} = \frac{c}{(c-1)(d-1)}$ if $u$ has degree $c$, or $\frac{d}{(c-1)(d-1)}$ if $u$ has degree $d$.  Then $\tilde D - \tilde\tau^2$ is diagonal with entry $\frac{c}{(c-1)(d-1)} - \frac{1}{(c-1)(d-1)} = \frac{1}{d-1}$ or $\frac{d}{(c-1))(d-1)} - \frac{1}{(c-1)(d-1)} = \frac{1}{c-1}.$  Hence the decomposition (\ref{eq:decomp})  becomes
\[
(I-u\tilde P)(I-u\tilde\tau) \sim \begin{bmatrix}I - u\begin{bmatrix}0&\frac{1}{c-1}M\\\frac{1}{d-1}M^T&0\end{bmatrix} + u^2\begin{bmatrix}\frac{1}{d-1}I&0\\0&\frac{1}{c-1}I\end{bmatrix} &*\\0&\left(1-\frac{u^2}{(c-1)(d-1)}\right) I\end{bmatrix}
\]
where $A = \begin{bmatrix}0&M\\M^T&0\end{bmatrix}$ is the adjacency matrix of $G$.

Note that $\tilde\tau$ is similar to a block diagonal matrix with blocks of the form $\begin{bmatrix}0&1/(c-1)\\1/(d-1)&0\end{bmatrix}$, so taking the determinant above we obtain
\[\begin{split}
\det(I - u\tilde P)\left(1-\frac{u^2}{(c-1)(d-1)}\right)^m &= \\\left(1-\frac{u^2}{(c-1)(d-1)}\right)^{2m-n}&\det\left(I - u\begin{bmatrix}0&\frac{1}{c-1}M\\\frac{1}{d-1}M^T&0\end{bmatrix} + u^2\begin{bmatrix}\frac{1}{d-1}I&0\\0&\frac{1}{c-1}I\end{bmatrix} \right)
\end{split}\]
so
\[
\det(I-u\tilde P) = \left(1+\frac{u^2}{(c-1)(d-1)}\right)^{m-n}\det\begin{bmatrix}\left(1+\frac{u^2}{d-1}\right)I&\frac{u}{c-1}M\\\frac{u}{d-1}M^T& \left(1-\frac{u^2}{c-1}\right)I\end{bmatrix}
\]
We will look at the matrix $\begin{bmatrix}\left(1+\frac{u^2}{d-1}\right)I&\frac{u}{c-1}M\\\frac{u}{d-1}M^T& \left(1+\frac{u^2}{c-1}\right)I\end{bmatrix}$. Suppose the first part in the bipartition of $G$ has size $r$, and the second part has size $s$, where without loss of generality, $r>s$.  By row reduction, this has the same determinant as the matrix
\[
\begin{bmatrix}\left(1+\frac{u^2}{d-1}\right)I&\frac{u}{c-1}M\\0& \left(1-\frac{u^2}{c-1}\right)I - \frac{1}{1+\frac{u^2}{d-1}}\frac{u^2}{(c-1)(d-1)}M^TM\end{bmatrix}
\]
which is
\[\begin{split}
&\left(1+\frac{u^2}{d-1}\right)^r\det\left(\left(1+\frac{u^2}{c-1}\right)I - \frac{1}{1+\frac{u^2}{d-1}}\frac{1}{(c-1)(d-1)}M^TM\right) \\ &= \left(1+\frac{u^2}{d-1}\right)^{r-s}\det\left(\left(1+\frac{u^2}{c-1}\right)\left(1+\frac{u^2}{d-1}\right)I -\frac{u^2}{(c-1)(d-1)}M^TM\right).
\end{split}\]

Now, the above determinant is given by the product of the eigenvalues of the matrix.  Observe that if $\lambda$ is an eigenvalue of the adjacency matrix $A$, then $\lambda^2$ is an eigenvalue of $M^TM$.  Therefore, in all we have
\[\begin{split}
\det(I-u\tilde P) =& \left(1-\frac{u^2}{(c-1)(d-1)}\right)^{m-n}\left(1+\frac{u^2}{d-1}\right)^{r-s}\\&\times\prod_{i=1}^s\left(\left(1+\frac{u^2}{c-1}\right)\left(1+\frac{u^2}{d-1}\right) - \frac{\lambda_i^2u^2}{(c-1)(d-1)}\right)
\end{split}\]
where the product ranges over the $s$ largest eigenvalues of $A$ (or in other words, $\lambda_i^2$ ranges of the $s$ eigenvalues of $M^TM$).  Therefore the characteristic polynomial is given by
\[\begin{split}
\det(uI - \tilde P) =& \left(u^2 - \frac{1}{(c-1)(d-1)}\right)^{m-n}\left(u^2 + \frac{1}{d-1}\right)^{r-s}\\&\times\prod_{i=1}^s\left(\left(u^2 + \frac{1}{c-1}\right)\left(u^2 + \frac{1}{d-1}\right) - \frac{\lambda_i^2u^2}{(c-1)(d-1)}\right).
\end{split}\]
Thus we can explicitly obtain the eigenvalues of $\tilde P$.

\begin{theorem}
Let $G$ be a $(c,d)$-biregular graph, let the part with degree $c$ have size $r$, and the part with degree $d$ have size $s$, and assume without loss of generality that $r\geq s$.  Suppose $G$ has $n$ vertices and $m$ edges.  Then the eigenvalues of the non-backtracking transition probability matrix $\tilde P$ defined above are
\[
\pm\frac{1}{\sqrt{(c-1)(d-1)}} \text{ with multiplicity }m-n \text{ each }, \, \pm\frac{1}{\sqrt{d-1}}i \text{ with multiplicity }r-s\text{ each }
\]
as well as the $4$ roots of the polynomial
\[
u^4 + \left(\frac{1}{(c-1)}+\frac{1}{(d-1)}-\frac{\lambda_i^2}{(c-1)(d-1)}\right)u^2 + \frac{1}{(c-1)(d-1)}
\]
for each value of $\lambda_i$ ranging over the $s$ positive eigenvalues of the adjacency matrix $A$.  These roots are
\begin{equation}\label{eigvals}
\pm\sqrt{\frac{\lambda_i^2 - (c-1) - (d-1)\pm\sqrt{(\lambda_i^2 - (c-1) - (d-1))^2-4(c-1)(d-1)}}{2(c-1)(d-1)}}, \, (i=1,\cdots,s).
\end{equation}
\end{theorem}

We can now give a version of Corollary \ref{thm:alon} for $(c,d)$-biregular graphs.

\begin{corollary}
Let $G$ be  a $(c,d)$-biregular graph with $c,d \geq 2$.  Let $\rho = \lambda^2/cd$ be the square of the second largest eigenvalue of the transition probability matrix $P$ for a random walk on $G$, and let $\tilde \rho = |\mu|^2$ be the square of the second largest modulus of an eigenvalue of $\tilde P$.  Let $\lambda$ be the second largest eigenvalue of the adjacency matrix of $G$.  Then we have the following cases.

If $\lambda > \sqrt{c-1}+\sqrt{d-1}$, then
\[
\frac{cd}{2(c-1)(d-1)}\left(1-\frac{c-1+d-1}{c-1+2\sqrt{(c-1)(d-1)}+d-1}\right)\leq \frac{\tilde \rho}{\rho} \leq 1.
\]

If $\lambda < \sqrt{c-1}+\sqrt{d-1}$ and both $c$ and $d$ are $n^{o(1)}$, then
\[
\frac{\tilde \rho}{\rho} \leq \frac{cd}{2(c-1)(d-1)} + o(1).
\]
\end{corollary}

\begin{proof}

We need to compare the eigenvalues of $\tilde P$ to the eigenvalues of $P = D^{-1}A = \begin{bmatrix}0&\frac1c M\\\frac1d M^T&0\end{bmatrix}$. Note that for $\lambda$ an eigenvalue of $A$, we have
\[
\begin{bmatrix}0&M\\M^T&0\end{bmatrix}\begin{bmatrix}x\\y\end{bmatrix} = \lambda\begin{bmatrix}x\\y\end{bmatrix}
\]
which implies $My = \lambda x$ and $M^T x = \lambda y$.  Then observe
\[
\begin{bmatrix}0&\frac1c M\\\frac1d M^T&0\end{bmatrix}\begin{bmatrix}\frac{1}{\sqrt c}x\\\frac{1}{\sqrt d}y\end{bmatrix} =\begin{bmatrix} \frac{1}{c\sqrt d}My\\\frac{1}{d\sqrt c }M^T x\end{bmatrix} = \frac{\lambda}{\sqrt{cd}}\begin{bmatrix}\frac{1}{\sqrt c}x\\\frac{1}{\sqrt d}y\end{bmatrix},
\]
so the eigenvalues of $P$ are $\lambda/\sqrt{cd}$ where $\lambda$ ranges over the eigenvalues of $A$.  Note that the largest eigenvalue of $A$ is $\sqrt{cd}.$

Let $\mu$ equal the expression (\ref{eigvals}), and consider the following cases.

If $\sqrt{c-1}+\sqrt{d-1} \leq \lambda \leq \sqrt{cd}$, then $\mu$ is real.  Direct computation verifies that, evaluating the expression (\ref{eigvals}) at $\lambda = \sqrt{cd}$ yields $\mu = 1 = \lambda/\sqrt{cd}$ and $\mu < \lambda/\sqrt{cd}$ for $\lambda$ in this range.  Therefore, in this case the eigenvalue of $\tilde P$ always has smaller absolute value than the corresponding eigenvalue of $P$, implying $\tilde\rho \leq\rho$. The lower bound follows from (\ref{eigvals}) ignoring the square root inside.  Thus the first case follows.  

If $\lambda < \sqrt{c-1}+\sqrt{d-1}$, then $\mu$ is complex, and direct computation shows 
\[
|\mu|^2 = \frac{1}{\sqrt{(c-1)(d-1)}},
\]
so 
\[
|\mu| = \frac{1}{((c-1)(d-1))^{1/4}}
\]

A version of the Alon-Boppana Theorem exists for $(c,d)$-biregular graphs as well, proven by Feng and Li in \cite{Feng_Li} (see also \cite{Li}).

\begin{theorem}[\cite{Feng_Li}]\label{thm:cd_alon_boppana}
Let $G$ be a $(c,d)$-biregular graph, and let $\lambda$ be the second largest eigenvalue of the adjacency matrix $A$ of $G$.  Then
\[
\lambda^2 \geq \left(\sqrt{c-1} + \sqrt{d-1}\right)^2 - \frac{2\sqrt{(c-1)(d-1)}-1}{k}
\]
where the diameter of $G$ is greater than $2(k+1)$.
\end{theorem}

Observe that certainly the diameter is at least $\log_{cd}n$, so that the condition on the degrees and Theorem \ref{thm:cd_alon_boppana} imply that \[\lambda^2 \geq 2\sqrt{(c-1)(d-1)}(1-o(1)).\]  As $|\mu|^2 = \frac{1}{\sqrt{(c-1)(d-1)}}$, so this gives the result for the second case.
\end{proof}

\section{Conclusion}
We have looked at non-backtracking random walks from the point of view of walking along directed edges.  For the special cases of regular and biregular graphs, our weighted version of Ihara's Theorem (Theorem \ref{thm:weighted_ihara}) has given us the comlete spectrum of the transition probability matrix for the non-bakctracking walk, allowing for easy comparison between the non-backracking mixing rate, and the mixing rate of the usual random walk.  Clearly, it would be desirable to extend these reults to more general classes of graphs.  The difficulty in applying Theorem \ref{thm:weighted_ihara} directly is with the term involving $\tilde\tau^2$.  As seen in section \ref{sec:ihara}, $\tilde\tau^2$ is a $2m\times 2m$ diagonal matrix with
\[
\tilde\tau^2\left((u,v),(u,v)\right) = \frac{1}{(d_u - 1)(d_v - 1)}.
\]
In the case of regular and biregular graphs, this expression is constant (we get $1/(d-1)$ and $1/(c-1)(d-1)$ for the $d$-regular and $(c,d)$-biregular cases respectively), making $\tilde\tau^2$ simply a multiple of the identity.  This allows the difficulty to be handled relatively easily.  Regular and biragular graphs are in fact the only graphs for which $\tilde\tau^2$ is a multiple of the identity, suggesting that these exact techniques will not work as nicely on more general classes of graphs.  If a cleaner version of Theorem \ref{thm:weighted_ihara} could be proven, then, aside from being interesting in its own rite, it could potentially be used to extend our results on non-backtracking random walks.  \\

{\bf Acknowledgment.} The author would like to thank Fan Chung for numerous helpful discussions throughout the process of writing this paper.

%----------------------------------------------------------------------------------------------------------------------

\end{document}